\documentclass{amsart}
\usepackage{amsfonts,amssymb,amscd,amsmath,enumerate,verbatim,calc}
\usepackage[all]{xy}

\newcommand{\CM}{Cohen-Macaulay}

\newcommand{\wrt}{with respect to}

\newcommand{\bF}{\mathfrak{b} }
\newcommand{\PP}{\mathcal{P} }
\newcommand{\n}{\mathfrak{n} }
\newcommand{\m}{\mathfrak{m} }

\newcommand{\q}{\mathfrak{q} }

\newcommand{\rt}{\rightarrow}

\newcommand{\ov}{\overline}

\newcommand{\rank}{\operatorname{rank}}

\newcommand{\depth}{\operatorname{depth}}
\newcommand{\mult}{\operatorname{mult}}

\newcommand{\add}{\operatorname{add}}

\newcommand{\Ass}{\operatorname{Ass}}

\newcommand{\Aut}{\operatorname{Aut}}
\newcommand{\Hom}{\operatorname{Hom}}
\newcommand{\Ext}{\operatorname{Ext}}

\theoremstyle{plain}

\newtheorem{theorem}{Theorem}[section]
\newtheorem{corollary}[theorem]{Corollary}
\newtheorem{lemma}[theorem]{Lemma}
\newtheorem{proposition}[theorem]{Proposition}

\theoremstyle{definition}

\newtheorem{s}[theorem]{}
\newtheorem{remark}[theorem]{Remark}

\theoremstyle{remark}

\begin{document}

\title[Finite representation type]{On two dimensional mixed characteristic rings of finite Cohen Macaulay  type}
\author{Tony~J.~Puthenpurakal}
\date{\today}
\address{Department of Mathematics, IIT Bombay, Powai, Mumbai 400 076}

\email{tputhen@math.iitb.ac.in}
 \begin{abstract}
In this paper we give a bountiful number of examples of two dimensional mixed characteristic rings of finite Cohen Macaulay  type. For a large sub-class of these examples we give a complete description of its indecomposable maximal \CM \ modules and we also compute its AR-quiver.
\end{abstract}
 \maketitle
\section{introduction}
Let $(A,\m)$ be a Henselian \CM  \ local ring of dimension $d \geq 0$. As $A$ is Henselian the category of finitely generated $A$-modules is Krull-Schmidt, i.e., any finitely generated $A$-module is uniquely a finite direct sum of indecomposable $A$-modules. We say $A$ has finite representaion type if $A$ has only finitely many indecomposable maximal \CM \ modules. 

There has been a lot of work towards understanding \CM \ rings of finite representation type. See \cite{Yoshino} for a very readable account of this work. 
We should note that although the basic theory is developed in general, most of the examples considered are equicharacteristic, i.e., $A$ contains a field. 
See \cite{GL} for examples of one-dimensional hypersurfaces of mixed characteristic rings of finite representation type.

Let $T = k[[x_1,x_2]]$ and let $G$ be a finite subgroup of $GL_2(k)$ acting linearly on $T$. In a fundamental work \cite{Aus-1}, Auslander proved that the ring of invariants $A = T^G$ is of finite representation type. When $G$ has no psuedo-reflections, he gave a description of all indecomposable maximal \CM \ $A$-modules and  constructed all AR-sequences of $A$. Furthermore he showed that the AR-quiver of $A$ is isomorphic to the McKay graph of $G$.

In this paper we construct examples of two dimensional mixed characteristic rings of finite Cohen Macaulay  type. Our examples also arise as invariant rings but with a twist.  Let $(V,\pi)$ be a complete DVR of characteristic zero having  residue field $k = V/\pi$, an algebraically closed field of characteristic $p > 0$. Let $G$ be a finite subgroup of $GL_2(V)$. We assume that $p \nmid |G|$. So $|G|$ is a unit in $V$. Let $x_1, x_2 $ be a basis of $V^2$ on which $G$ naturally acts. Then the action of $G$ can be extended to the ring $S = V[[x_1,x_2]]$.  Set $\m = (\pi, x_1, x_2)$, the maximal ideal of $S$. Let $S^G$ be the ring of invariants of $S$ \wrt \ $G$. Let $f \in (x_1,x_2)^2\cap S^G$. We also assume $f \notin \pi S$. Set $R = S/(\pi - f)$. Notice $\pi - f \in \m\setminus \m^2$. So $R$ is a regular local ring of dimension two. Note 
$G$ acts on $R$. It can be shown that $R^G$ has finite representation type, see \ref{herzog}. By \ref{mod-reg} we get that $R^G \cong S^G/(\pi -f)$. It can be easily shown that $R$ is of mixed characteristic. In Proposition \ref{construct-groups} we show that there are bountiful number of finite subgroups $G$ of 
$GL_2(V)$ with $p \nmid |G|$. Thus there are lots of examples of two dimensional mixed characteristic rings of finite Cohen Macaulay  type.

A natural question regarding rings of finite representation type is to describe its maximal \CM \ modules and find its AR-quiver. It is perhaps hopeless to this to this to all examples of $R^G$ above. However we are able to do this for a large subclass of the examples given above. 
Recall $\sigma \in GL_2(k)$ is said to be a pseudo-reflection if $\rank(\sigma - 1) \leq 1$. Let $G$  be a finite subgroup of $GL_2(V)$ with $p \nmid |G|$. By \ref{inject-k}  the natural map  $\eta \colon G \rt GL_2(k)$ is an inclusion. We say $G$ has no pseudo-reflections except the identity if $\eta(g)$ is not a pseudo-reflection for all $g \neq 1$. 
Let $\PP(V[G])$ be the category of finitely generated projective $V[G]$-modules and let $CM(R^G)$ be the category of maximal \CM \ $R^G$-modules. In \ref{psi-def} we construct a functor $\psi \colon \PP(V[G]) \rt CM(R^G)$. 
We   prove
\begin{theorem}\label{basic-int}(with hypotheses as above) Assume $G \subseteq GL_2(V)$ has no pseudo-reflection except the identity. Let $f \in (x_1,x_2)^l \setminus \pi S$. There  exists a positive integer $l_0(G)$ depending on $G$ such that if $l \geq l_0(G)$ then the functor $\psi \colon \PP(V[G]) \rt CM(R^G)$ has the following properties
\begin{enumerate}[\rm (1)]
\item
If $P$ is an indecomposable projective $V[G]$-module then $\psi(P)$ is an indecomposable maximal \CM \ $R^G$-module.
\item
$P_1 \cong P_2$ in $\PP(V[G])$ if and only if $\psi(P_1) \cong \psi(P_2)$ as $R^G$-modules.
\item
If $M$ is an indecomposable maximal \CM \ $R^G$-module then there exists an indecomposable projective $V[G]$-module $P$ with $\psi(P) \cong M$ as $R^G$-modules.
\end{enumerate}
\end{theorem}
We then in Theorem \ref{ar-main} construct all AR-sequences of $R^G$. Finally in Theorem \ref{AR-quiver} we  show that the AR-quiver of $R^G$ is isomorphic to the McKay graph of $G$.

We now describe in brief the contents of this paper. In section two we discuss a few generalities which we need. In section three we construct our examples of mixed characteristic rings of dimension two having finite representation type. We also give many examples of finite groups of $GL_2(V)$ with $p \nmid |G|$. 
In the next section we construct a few functors that we need. In section 5 we prove Theorem \ref{basic-int}. In section 6 we discuss a few preliminaries we need to construct the AR-sequences of $R^G$. In the next section we determine all AR sequences of $R^G$ and we show that the AR-quiver of $R^G$ is isomorphic to the McKay graph of $G$. Finally in section 8 we give an explicit example which illustrates our results. 

\section{generalities}
All the results in this section are either known or easy extensions of  known results. However for the convenience of the reader we give a sketch of a proof of all the results in this section. The main goal of this section is to prove a generalization of a theorem due to Herzog \cite[1.7]{Herzog}.

In this section $(R,\m)$ will denote a Noetherian local ring.  Let $\Aut(R)$ be the group of automorphism's  of $R$. We note that if $f \in \Aut(R)$ then it is local, i.e., $f(\m) \subseteq \m$. 
 Let $G$ be a finite group and let $ \eta\colon G \rt \Aut(R)$ be a group homomorphism. In this case we say $G$ \textit{ acts on} $R$. We assume that $|G|$ is a unit in $R$.
 Let
 $$R^G = \{ x \in R \mid \sigma(x) = x \  \text{for all} \ \sigma \in G \},$$
 be the ring of invariants of $G$. It is easy to see that $R^G$ is local with maximal ideal $\n = \m \cap R^G$. Clearly $R$ is integral over $R^G$. We  have a Reynolds operator
 $\rho \colon R \rt R^G$ defined as
 $$\rho(x) = \frac{1}{|G|} \sum_{\sigma \in G} \sigma(x).$$
 It is easily verified that $\rho$ is $R^G$-linear. Using the Reynolds operator it can be easily seen  that if $I$ is an ideal in $R^G$ then
 $IR \cap R^G = I$. It follows that $R^G$ is a Noetherian ring.

\begin{lemma}\label{top}
 For $j \geq 1$ set $\n_j = \m^j \cap R^G$. Then the filtration $\{\n_j\}$ defines on $R^G$ the same topology as the $\n$-adic topology.
\end{lemma}
\begin{proof}(Sketch)
 Clearly $\n^j \subseteq \n_j$. Also $\n R$ is $\m$-primary. Say $\m^s \subseteq \n R$. Then notice for every $t \geq 1$ we have
 $$ \n_{st} = \m^{st}\cap R^G \subseteq \n^tR \cap R^G = \n^t. $$
 \end{proof}
Our next result considers the case when $R$ is complete.
\begin{proposition}\label{complete}
If $R$ is complete \wrt\ $\m$ then $R^G$ is complete \wrt \ $\n$.
\end{proposition}
 \begin{proof}(Sketch)
 Define $\widetilde{\rho} \colon R \rt R$ to be $\widetilde{\rho}(x) = \rho(x)$.   Each $\sigma \in G$ is a local map. So $\widetilde{\rho}$ is continuous. Using \ref{top} it follows that $\rho$ is continuous. Let $\{x_n \}$ be a Cauchy sequence in $R^G$. Since $\n^j \subseteq \m^j$ for all $j \geq 1$ we get that $\{ x_n \}$ is a Cauchy sequence in $R$. As $R$ is complete $\{ x_n \}$ is convergent, say it converges to $x$. As $\widetilde{\rho}$ is continuous we get that
 $$ x_n = \widetilde{\rho}(x_n) \rt \widetilde{\rho}(x).$$
 So $\rho(x) = x$, i.e., $x \in R^G$. As $\rho$ is continuous we get that
 $\rho(x_n) \rt \rho(x)$ in $R^G$. So $x_n \rt x$ in $R^G$.
\end{proof}
In general it is not clear whether $R$ is finitely generated as a $R^G$-module. However when $R$ is complete we have the following result:
\begin{theorem}\label{finite-gen}
Let $(R,\m)$ be a complete Noetherian local ring and let $G$ be a finite group  acting on $R$.  Assume $|G|$ is invertible in $R$. Then $R$ is a finite $R^G$-module.
\end{theorem}
An essential ingredient to prove Theorem \ref{finite-gen} is the following:
\begin{lemma}\label{finite-lenght}
(with hypotheses as above) Let $E$ be an $R$-module of finite length. Then $E$ is    finitely generated as a $R^G$-module.
\end{lemma}
\begin{proof}
By an easy induction  it suffices to prove $k = R/\m$ is    finitely generated as a $R^G$-module.
Note the action of $G$ on $R$ induces an action on $k$. Let $k^G$ denote the fixed field of this action.  Set $\n = R^G \cap \m$. Note we have a natural map $i \colon R^G/\n \rt R/\m$.

\textit{Claim}:  $i(R^G/\n) = k^G$.

Let $t = [\xi] \in k^G$. Then $\sigma(t) = t $ for every $\sigma \in G$. So
$\sigma(\xi) = \xi + v_{\xi,\sigma} $ for some $v_{\xi,\sigma} \in \m$. So
$\rho(\xi) = \xi + v_{\xi}$ for some $v_{\xi} \in \m$. Set $\theta = \rho(\xi) \in R^G$. Clearly $i([\theta]) = [\xi] = t$.

By a result due to Artin,  $k$ is a finite extension of $k^G$, cf. \cite[Chapter 6,  Theorem 1.8]{Lang}. It follows that $k$ is a finite $R^G$-module.
\end{proof}
We now give
\begin{proof}[Proof of Theorem \ref{finite-gen}]
Set $\n = \m \cap R^G$. By \ref{complete} $R^G$ is complete \wrt \ $\n$.  Notice
$$ \bigcap_{j \geq 1} \n^j R \subseteq \bigcap_{j \geq 1} \m^j = 0.$$
So $R$ is separated \wrt\ the $\n$-adic topology. Also note that $R/\n R$ has finite length. So by \ref{finite-lenght} $R/\n R$ is finitely generated as a $R^G$-module.
Thus by \cite[8.4]{Mat},  $R$ is finitely generated as a $R^G$-module.
\end{proof}

We now extend a result of Herzog, \cite[1.7]{Herzog}, with nearly the same proof.
\begin{theorem}\label{herzog}
Let $(R,\m)$ be a two dimensional complete regular local ring and let $G$ be a finite group  acting on $R$.  Assume $|G|$ is a unit in $R$. Then $R^G$ is a normal \CM \
domain of dimension two. Furthermore $R^G$ is of finite \CM \ type.
\end{theorem}
\begin{proof}
By \cite[6.4.1]{BH}, $R^G$ is a normal domain. As $R$ is a finite extension of $R^G$ we have $\dim R^G = 2$. So $R^G$ is \CM. It can be easily verified that $R$ is a MCM  $R^G$-module.

As $R^G$ is normal, MCM $R^G$-modules are reflexive. Set $(-)^* = \Hom_{R^G}(-,R^G)$.
Let $M$ be an MCM $R^G$-module. The inclusion $i \colon R^G \rt R$ is a split map as $R^G$-modules. It follows that $\Hom_{R^G}(M^*, R^G)$ is a direct summand of $\Hom_{R^G}(M^*, R)$. But $\Hom_{R^G}(M^*, R)$ has depth $2$ as a $R$-module. So it is free as a $R$-module. Thus $M \cong \Hom_{R^G}(M^*, R^G)$ is a direct  summand (as a $R^G$-module) of some copies of $R$. It follows that if $M$ is an indecomposable MCM $R^G$-module then it is a direct summand of $R$. Thus $R^G$ is of finite \CM \ type.
\end{proof}

We end this section by an elementary result which is crucial in our paper.
\begin{proposition}\label{mod-reg}
Let $(R,\m)$ be a  Noetherian local domain and and let $G$ be a finite group
 acting on $R$.
   Assume $|G|$ is invertible in $R$. Let $x \in R^G$ be non-zero. Set $T = R/(x)$. Note $G$ acts on $T$. Then $T^G \cong R^G/x R^G$.
\end{proposition}
\begin{proof}
The inclusion map $i \colon R^G \rt R$ is split as $R^G$-modules. Thus $\ov{i} \colon R^G/x R^G  \rt T$ is an inclusion. Clearly  $\ov{i}(R^G/x R^G) \subseteq T^G$. Conversely let $t = [\xi] \in T^G$. Then $\sigma(\xi) = \xi + xv_{\xi,\sigma}$ for some  $v_{\xi,\sigma} \in R$. Note $\rho(\xi) = \xi + xv_{\xi}$ for some $v_{\xi} \in R$. Notice $\ov{i}([\rho(\xi)]) = t$. Thus $\ov{i}(R^G/x R^G) = T^G$.
\end{proof}
\section{Construction of two dimensional mixed characteristic rings of finite Cohen Macaulay  type}
In this section $(V,\pi)$ is a complete DVR of characteristic zero having  residue field $k = V/\pi$, an algebraically closed field of characteristic $p > 0$. Let $G$ be a finite subgroup of $GL_2(V)$. We assume that $p \nmid |G|$. So $|G|$ is a unit in $V$. We construct examples of two dimensional mixed characteristic rings of finite Cohen Macaulay  type. We also give ample number of finite subgroups of $GL_2(V)$ with $p \nmid |G|$.

\s Let $x_1, x_2 $ be a basis of $V^2$ on which $G$ naturally acts. Then the action of $G$ can be extended to the ring $S = V[[x_1,x_2]]$. It is clear that $G$  acts on $S$ via local automorphisms. Set $\m = (\pi, x_1, x_2)$, the maximal ideal of $S$. Let $S^G$ be the ring of invariants. By results in the previous section $S^G$ is local with maximal ideal $\n = \m \cap S^G$. Furthermore $S^G$ is complete \wrt \ the $\n$-adic topology. Also $S$ is finitely generated as a $S^G$-module.

\s Let $f \in (x_1,x_2)^2\cap S^G$. We also assume $f \notin \pi S$. Set $R = S/(\pi - f)$. Notice $\pi - f \in \m\setminus \m^2$. So $R$ is a regular local ring of dimension two. It is clear that the maximal ideal $\q$ of $R$ is generated by images of $x_1,x_2$ in $R$. Also note that $\pi - f \in S^G$. So by Proposition \ref{mod-reg} we have that $R^G = S^G/(\pi-f)S^G$. By Theorem \ref{herzog} we have that $R^G$ has finite \CM \ type.

\s \label{natural-map} Notice the natural map $j \colon V \rt R$ is an inclusion, for if $\pi^l \in \ker j$ then $\pi^l \in (\pi - f)$. Then there exists $g \in S$ with $\pi^l = (\pi - f)g$.
As $S$ is a UFD we have that $\pi - f = u\pi^r$ where $u$ is a unit in $S$. This implies that $f \in \pi S$ a contradiction. Also note that $R/\q = k$ has characteristic $p$. So $R$ is of mixed characteristic. It also follows that $R^G$ is of mixed characteristic.

\s Note that there is a natural map $\eta \colon G \rt GL_2(k)$. We prove
\begin{proposition}\label{inject-k}
The map $\eta$ is an inclusion.
\end{proposition}
\begin{proof}
We consider elements of $GL_2(V), GL_2(k)$ as matrices. Let $T \in G$ with $\eta(T) = I$. Then $T = I + Y$ where all entries of $Y$ are in $(\pi)$. Let $m = |G|$. Then
$T^m = I$. So we have
\[
mY + \binom{m}{2}Y^2 + \cdots+ \binom{m}{m-1}Y^{m-1} + Y^m = 0.
\]
Set
$$ U = mI + \binom{m}{2}Y + \cdots+ \binom{m}{m-1}Y^{m-2} + Y^{m-1}.$$
Note $\eta(U) = m I$ is invertible in $GL_2(k)$ since $p \nmid m$. By Nakayama Lemma it follows that $U \colon V^2 \rt V^2 $ is surjective and hence an isomorphism. So $U \in GL_2(V)$.

As $YU = 0$, we get that $Y = 0$. It follows that $\eta$ is injective.
\end{proof}

We now give ample number  of finite subgroups of $GL_2(V)$ with $p \nmid |G|$.
\begin{proposition}\label{construct-groups}
Let $H$ be a finite subgroup of $GL_n(\mathbb{C})$ with $n \geq 1$. Assume $p \nmid m$, where $m = |H|$. Then there is a finite subgroup $G$ of $GL_n(V)$ with $G \cong H$ as groups. Furthermore if $H \subseteq SL_n(\mathbb{C})$ then $G \subseteq SL_n(V)$.
\end{proposition}
\begin{proof}
Let $\zeta_m$ denote a primitive $m^{th}$-root of unity in $\mathbb{C}$.
By a result due to Brauer $L = \mathbb{Q}(\zeta_m)$ is a splitting field of $H$, (see \cite{BR}, also see \cite[41.1, p.\ 292]{CR}). It follows that $H$ is conjugate to a subgroup $H^\prime$ which is contained in $GL_n(L)$.

As $k$ is algebraically closed we have a primitive $m^{th}$ root of unity in $k$, say $t$. As  $p \nmid m$ it follows that $1, t, t^2, \ldots, t^{m-1}$ are all distinct. By Hensel's Lemma there exists $\theta \in V$ with $\ov{\theta} = t $ and $\theta^m = 1$.

Let $K =$ quotient field of $V$. As $V$ has a primitive $m^{th}$-root of unity and as
$\mathbb{Q} \subseteq K$ we have an embedding $L \rt K$. Thus $H^\prime$ is isomorphic to $G^\prime$ for some $G^\prime \subseteq GL_n(K)$.
As $V$ is a P.I.D we get that $G^\prime $ is conjugate to a group $G \subseteq GL_n(V)$, see \cite[73.6, p.\ 496]{CR}.

Each step of our construction preserves the determinant. So if $H \subseteq SL_n(\mathbb{C})$ then $G \subseteq SL_n(V)$.
\end{proof}
\section{Construction of a few functors}\label{c-f}
In this section we define a few functors which are analogous to
those defined by Auslander in \cite{Aus-1}. Let $(V,\pi)$ be a complete
DVR of characteristic zero such that $V/(\pi) = k$ is an
algebraically closed field of characteristic $p
> 0$. Let $G \subseteq GL_2(V)$ be a finite group with $p \nmid
|G|$. Let $S = V[[x_1,x_2]]$ and let $G$ act linearly on $S$. Let
$S^G$ be the ring of invariants of $S$ \wrt \ $G$. Let $f \in (x_1,x_2)^2\cap S^G$.
Assume $f \notin \pi S$. Let $R = S/(\pi-f)$. Note that $R$ is regular
local of dimension $2$ and $G$ acts on $R$. Let the ring of
invariants be $R^G$. By \ref{mod-reg} $R^G = S^G/(\pi-f)$. Let $A =
S\otimes_V k = k[[X_1,X_2]]$. Then we have a natural $G$-action on
$A$. Let $A^G$ be the ring of invariants of $A$ \wrt \ $G$.

\s Let $S*G$ be the skew-group ring of $S$ with respect to $G$.
Recall $S*G = \{ \sum_{\sigma \in G} a_{\sigma}\sigma \mid
a_{\sigma} \in S \}$ with multiplication defined by
\[
a_1\sigma_1\cdot a_2 \sigma_2 = a_1\sigma_1(a_2)\sigma_1\sigma_2.
\]
An $S*G$ module $M$ is precisely an $S$-module $M$ on which $G$ acts
such that $\sigma(a m) = \sigma(a)\sigma(m)$ for all $a \in S$ and
$m \in M$.

 As $|G|$ is invertible in $S$ taking invariants is an exact
 functor. It follows that an $S*G$ module
 $M$ is projective as an $S*G$-module if and only if it is
 projective as a $S$-module,( the proof in \cite[Lemma 1.1]{Aus-1} generalizes).

 \s If $T$ is a ring then we let $\PP(T)$ denote the category of
finitely generated projective left $T$-modules. Let $V[G]$ denote
the group ring over $V$.  Our first functor is
\begin{align*}
F \colon \PP(V[G]) &\rt \PP(S*G) \\
     W &\rt S\otimes_V W \\
     f &\rt 1_S \otimes f.
\end{align*}
Here $f$ is a morphism between two projective $V[G]$-modules. We
first note that $S\otimes_V W$ is a $S*G$-module. Clearly it is a
$S$-module. We define a $G$-action as follows: $\sigma( s\otimes w)
= \sigma(s) \otimes \sigma(w)$. Note that for $a \in S $ and $m =
s\otimes w \in S\otimes_V W$ we have
\[
\sigma(a m) = \sigma( as \otimes w) = \sigma(as)\otimes \sigma(w) =
\sigma(a)\sigma(s)\otimes \sigma(w) = \sigma(a)\sigma(m).
\]
Thus $S\otimes W$ is a $S*G$-module. As $W$ is a projective
$V[G]$-module it is free as a $V$-module. So $S\otimes_V W$ is free
as a $S$-module. It follows that $S\otimes_V W$ is a projective
$S*G$-module.

\s Let $\ov{F} \colon \PP(k[G]) \rt \PP(A*G)$ be the functor
$A\otimes_k -$ defined analogously as before. Note that as $k[G]$ is
semi-simple we have $\PP(k[G])$ is the category of all finitely
generated $k[G]$-modules. 

\s We first note that $V \subseteq Z(V[G])$, the center of $V[G]$. So we have an isomorphism of rings $k \otimes_V V[G] \cong k[G]$. We have a natural functor $k\otimes_V -
\colon \PP(V[G]) \rt \PP(k[G])$. To see this note that if $P$ is a finitely generated  $V[G]$-module then $k\otimes_V P$ is a finitely generated $k[G]$-module. It is also projective as $k[G]$ is semi-simple.

\s We also note that $V \subseteq Z(S*G)$, the center of $S*G$. So  we have an isomorphism of rings
$k \otimes_V S*G \cong A*G$. We have a natural functor $k\otimes_V - \colon \PP(S*G) \rt \PP(A*G)$.
To see this note that  if $M$ is a $S*G$-module then clearly $k\otimes_V M$ is an 
$A*G$-module.
  If $P$ is projective as $S*G$-module then it is free as a $S$-module. Hence $k\otimes_V P$ is free as an $A$-module. It  follows that $k \otimes_V P$ is a projective $A*G$-module.

\s We have a commutative diagram of functors
\[
\xymatrix
{
& \PP(V[G])
\ar@{->}[r]^{k\otimes_V-}
\ar@{->}[d]^{F}
& \PP(k[G])
\ar@{->}[d]^{\ov{F}}
\\
& \PP(S*G)
\ar@{->}[r]^{k\otimes_V-}
& \PP(A*G).
}
\]

\s Let $P$ be a projective $S*G$-module. Then $P^G$ is a $S^G$-direct summand of $P$. Also note that as $P$ is free as a $S$-module, we have that $P^G \in \add_{S^G} S$.
If $f \colon P_1 \rt P_2$ is a morphism in $\PP(S*G)$ then notice that we have a morphism $\widetilde{f} \colon P_1^G \rt P_2^G$ where $\widetilde{f}$ is the restriction map. Thus we have a functor 
\[
(-)^G \colon \PP(S*G) \rt \add_{S^G}(S).
\]
Similarly we have a functor
\[
(-)^G \colon \PP(A*G) \rt \add_{A^G}(A).
\]

\s Notice that $A^G = (S/(\pi))^G \cong S^G/(\pi)$. Thus $k \otimes_V S^G = A^G$. Furthermore it is clear that if $M \in \add_{S^G}(S)$ then $k\otimes_V M = M/\pi M \in \add_{A^G}(A)$. Thus we have a functor 
$$k\otimes_V- \colon \add_{S^G}(S) \rt \add_{A^G}(A). $$

\s We have a commutative diagram of functors
\[
\xymatrix
{
& \PP(S*G)
\ar@{->}[r]^{k\otimes_V-}
\ar@{->}[d]^{(-)^G}
& \PP(A*G)
\ar@{->}[d]^{(-)^G}
\\
& \add_{S^G}(S)
\ar@{->}[r]^{k\otimes_V-}
& \add_{A^G}(A).
}
\]

\s Notice $R^G = S^G/(\pi - f)$. Also note that $R = S/(\pi -f)$. Thus we have a functor
\[
R^G\otimes_{S^G} - \colon \add_{S^G}(S) \rt \add_{R^G}(R).
\]
By \ref{herzog}, we have that $\add_{R^G}(R) = CM(R^G)$ the category of all maximal \CM \ $R^G$-modules. Similarly $\add_{A^G}(A) = CM(A^G)$. As $\ov{f}$ is a non-zero divisor in $A^G$ we have a functor
\[
A^G/(\ov{f})\otimes_{A^G} - \colon CM(A^G) \rt CM(A^G/(\ov{f})).
\]

\s By \ref{natural-map} we have that $V \subseteq R$. Thus $V \subseteq R^G$. It follows that $\pi$ is a  non-zero element of $R^G$ and so a non-zero divisor as $R^G$ is a domain. Thus we have a functor
\[
k\otimes_V-\colon \add_{R^G}(R) \rt CM(A^G/(\ov{f})).
\] 

\s We have a commutative diagram of functors
\[
\xymatrix
{
& \add_{S^G}(S)
\ar@{->}[r]^{k\otimes_V-}
\ar@{->}[d]^{R^G\otimes_{S^G} -}
& \add_{A^G}(A)= CM(A^G)
\ar@{->}[d]^{A^G/(\ov{f})\otimes_{A^G} -}
\\
& \add_{R^G}(R) = CM(R^G)
\ar@{->}[r]^{k\otimes_V-}
& CM(A^G/(\ov{f}))
}
\]
\s\label{psi-def} We set $\psi \colon \PP(V[G]) \rt CM(R^G)$ to be the composite functor, i.e.,
\[
\psi = \left( R^G\otimes_{S^G}- \right) \circ  (-)^G   \circ   F. 
\]

\section{A property of the functor $\psi$}
Let $\psi \colon \PP(V[G]) \rt CM(R^G)$ be the functor as defined in the previous section. In this section we prove a crucial result of $\psi$. The hypotheses in this section is similar to the previous two sections. We need an additional hypotheses.
Recall $\sigma \in GL_2(k)$ is said to be a pseudo-reflection if $\rank(\sigma - 1) \leq 1$. Let $G$  be a finite subgroup of $GL_2(V)$ with $p \nmid |G|$. By \ref{inject-k}  we also have a natural inclusion $\eta \colon G\hookrightarrow GL_2(k)$. We say $G$ has no pseudo-reflections except the identity if $\eta(g)$ is not a psuedo-reflection for all $g \neq 1$. In this section we prove Theorem \ref{basic-int}. We restate it here for the convenience of the reader. 
\begin{theorem}\label{basic}(with hypotheses as above) Assume $G \subseteq GL_2(V)$ has no pseudo-reflection except the identity. Let $f \in (x_1,x_2)^l \setminus \pi S$. There  exists a positive integer $l_0(G)$ depending on $G$ such that if $l \geq l_0(G)$ then the functor $\psi \colon \PP(V[G]) \rt CM(R^G)$ has the following properties
\begin{enumerate}[\rm (1)]
\item
If $P$ is an indecomposable projective $V[G]$-module then $\psi(P)$ is an indecomposable maximal \CM \ $R^G$-module.
\item
$P_1 \cong P_2$ in $\PP(V[G])$ if and only if $\psi(P_1) \cong \psi(P_2)$ as $R^G$-modules.
\item
If $M$ is an indecomposable maximal \CM \ $R^G$-module then there exists an indecomposable projective $V[G]$-module $P$ with $\psi(P) \cong M$ as $R^G$-modules.
\end{enumerate}
\end{theorem}

\s\label{perfect} $V$ is a complete DVR. So the group ring $V[G]$ is semi-perfect, see \cite[23.3]{Lam}.
In particular the functor $k\otimes_V- \colon \PP(V[G]) \rt \PP(k[G])$ gives a one-to-one correspondence of indecomposable projective modules, see \cite[25.3]{Lam}.

\s \label{efficient sop} Let $A = k[[x_1,x_2]]$ and let $A^G$ be the ring of invariants of $A$ \wrt \ $G$. Then there exists an efficient system of parameters $(u_1,u_2)$ of $A^G$, for this notion see \cite[6.14]{Yoshino}. In particular we have that  a maximal \CM \ $A^G$-module $M$ is  indecomposable  if and only  if $M/(u_1^2,u_2^2)M$ is indecomposable, see \cite[6.16]{Yoshino}. Furthermore $M_1 \cong M_2$ if and only if $M_1/(u_1^2,u_2^2)M_1 \cong M_2/(u_1^2,u_2^2)M_2$; see \cite[6.18]{Yoshino}.

Let $\n$ be the maximal ideal of $A^G$. Assume $\n^m \subseteq (u_1^2,u_2^2)$. By \ref{top} there exists $l$ such that $(x_1,x_2)^l\cap A^G \subseteq \n^m$. We define $l_0(G; u_1,u_2)$ to be the smallest integer $l$ with $(x_1,x_2)^{l}\cap A^G  \subseteq (u_1^2,u_2^2)$. Define
\[
l_0(G) = \min \{ l_0(G; u_1,u_2) \mid u_1, u_2 \ \text{is an efficient system of parameters of $A^G$} \}.
\]

\s \label{Aus} We recall some results of Auslander, \cite{Aus-1}. The functor $\ov{F} \colon \PP(k[G]) \rt \PP(A*G)$ gives an one-to-one correspondence between indecomposable objects. If $G$ has no pseudo-reflections except the identity then $(-)^G \colon \PP(A*G) \rt \add_{A^G}(A)$ is an equivalence of categories. 

We now give
\begin{proof}[Proof of Theorem \ref{basic}]
We take $l_0 = l_0(G)$, the invariant of $G$ defined in \ref{efficient sop}. Assume
$f \in (x_1,x_2)^l\cap S^G$ where $l \geq l_0$. Let $\ov{f}$ be the image of $f$ in $A^G$. By construction there exists an efficient system of parameters $u_1, u_2$ of $A^G$ such that $\ov{f} \in (u_1^2,u_2^2)$.

(1) Let $P$ be an indecomposable projective $V[G]$-module. Then $\ov{P} \in \PP(k[G])$ is indecomposable.
So $L = \ov{F}(\ov{P}) \in \PP(A*G)$ is indecomposable. It follows that $L^G$ is an indecomposable maximal \CM \ $A^G$-module. Note that $\ov{f} \in (u_1^2,u_2^2)$. So $L^G/\ov{f}L^G$ is an indecomposable maximal \CM \ $A^G/(\ov{f})$-module. By the commutaivity of the functors we constructed in the previous section we get that
 $k\otimes_V\psi(P) = L^G/\ov{f}L^G$. It follows that $\psi(P)$ is indecomposable. 

(2) If $P_1 \cong P_2$ then $\psi(P_1) \cong \psi(P_2)$. Conversely assume that $\psi(P_1)\cong \psi(P_2)$. Then $k\otimes_V\psi(P_1) \cong k\otimes_V\psi(P_2)$ as $A^G/(\ov{f})$-modules. Let 
$$M_i =  \left(\ov{F}(\ov{P_i}) \right)^G \quad \text{for} \ i = 1,2; $$
be objects in $CM(A^G)$. By the commutaivity of the functors we constructed in the previous section we get that $M_1/\ov{f}M_1 \cong M_2/\ov{f}M_2$. As $\ov{f} \in (u_1^2,u_2^2)$ we get that $M_1 \cong M_2$. By \ref{Aus} we get that $\ov{P_1} \cong \ov{P_2}$ as $k[G]$-modules. By \ref{perfect} we get that $P_1 \cong P_2$.

(3) In \cite[10.9]{Yoshino} it is proved that $(A*G)^G \cong A$ as $A^G$-modules. The same proof yields that $(S*G)^G \cong S$ as $S^G$-modules. Let $V[G] = P_1 \oplus P_2 \oplus \cdots \oplus P_m$ where $P_i$ are indecomposable projective $V[G]$-modules. Notice $F(V[G]) = S*G$. It follows that $\psi(V[G]) \cong R$ as $R^G$-modules. Thus
$$R = \psi(P_1)\oplus \psi(P_2)\oplus \cdots \oplus \psi(P_m). $$
By (1) we get that each  $\psi(P_i)$ is an indecompsable maximal \CM \ $R^G$-module. By the proof of \ref{herzog} we get that $M$ is isomorphic to $\psi(P_i)$ for some $i$. 
\end{proof}
\section{Preliminaries to construct AR sequences}
We assume $f \in (x_1,x_2)^l\cap S^G$ with $l \geq l_0(G)$. Notice we are assuming that $G$ has no pseudo-reflections. 
We need several preliminary results to construct AR-sequences on $R^G$.

We begin with following well-known fact for which I do not have a reference.

\begin{proposition}\label{ht-1}
Let $G \subseteq GL_2(V)$ be a finite group with $p \nmid |G|$. Assume $G$ has no pseudo-reflections. Let $G$ act linearly on $S = V[[x_1,x_2]]$. Then height one primes in $S^G$ are unramified in $S$.
\end{proposition}
\begin{proof}
Let $\bF$ be a height one prime in $S$. Put $\q = \bF \cap S^G$.  As $S^G$ is normal we have that $S_\q$ is a DVR. Set
\begin{align*}
T_\bF &= \{ \sigma \in G \mid \sigma(a) - a \in \bF \ \text{for all} \ a \in S \}, \text{and} \\
T_{\bF S_\q} &= \{ \sigma \in G \mid \sigma(x) - x \in \bF S_\q \ \text{for all} \ x \in S_\q \}.
\end{align*}

To show that $\q$ is unramified in $S$ it suffices to 
 prove $\sharp T_{\bF S_\q} = 1$, see \cite[Chapter 1, Propositions 20,21]{Serre}.
It can be easily shown that  
\[
T_\bF = T_{\bF S_\q}.
\]
So we prove $\sharp T_{\bF } = 1$.

$S$ is a unique factorization domain. So $\bF = (z)$. Let $\n = (\pi, x_1, x_2)$ be the maximal ideal of $S$. Let $A = S/(\pi) = k[[x_1,x_2]]$. Let $\m$ be the maximal ideal of $A$.
We consider two cases.

Case 1: $z \in \n^2$.\\
Then as $\sigma$ acts trivially on $S/\bF S$ we get that $\sigma$ acts trivially on $A/\m^2 A$. So $\sigma$ acts trivially on $\m/\m^2 = kx_1\oplus kx_2$. By \ref{inject-k} it follows that $\sigma = 1$.

Case 2. $z \in \n \setminus \n^2$. \\ 
Say  $z = a_0 \pi + a_1 x_1 + a_2x_2$. Put $u = a_1x_1 + a_2x_2$. Set $\ov{u}$ to be the image of $u$ in $A$. We again consider the following two subcases.

Subcase 1. $\ov{u} \in \m^2$. \\
As $\sigma$ acts trivially on $S/\bF S$, it acts trivially on $S/(\bF, \pi) = A/(\ov{u})$. As $\ov{u} \in \m^2$ we get that  $\sigma$ acts trivially on $A/\m^2 A$. So $\sigma$ acts trivially on $\m/\m^2 = kx_1\oplus kx_2$. By \ref{inject-k} it follows that $\sigma = 1$.

Subcase 2. $\ov{u} \in \m \setminus \m^2$. \\
As before $\sigma$ acts trivially on  $A/(\ov{u})$. So $\sigma - 1$ is null on $\m/(\ov{u} + \m^2)$. It follows that $\rank(\sigma - 1) \leq 1$ on $kx_1 \oplus k x_2$, i.e., $\sigma$ is a pseudo-reflection. By hypothesis $\sigma = 1$.
\end{proof}

\s Proposition \ref{ht-1} plays a significant role in all further analysis. The crucial point is a result of Auslander and Reiten which we now state. 
Let $E$ be a finitely generated  reflexive, left $S*G$-module. Then $E^G$  is a reflexive $S^G$-module, see \cite[Part 2, 1.1]{Aus-3}. Let $Ref(T)$ be the category of reflexive left modules over a ring $T$.  As height one primes in $S^G$ are unramified in $S$, the fixed point functor $(-)^G \colon mod(S*G) \rt mod(S^G)$ induces an equivalence of categories between $Ref(S*G)$ and $Ref(S^G)$, see \cite[Part 2, 1.3]{Aus-3}. 
If $P$ is a projective $S*G$-module then it is clearly a reflexive $S*G$-module. So we get the following result:
\begin{corollary}\label{consequence-ht1}(with hypotheses as above)
Let $P_1,P_2$ be projective $S*G$-modules. Let $M_i = P_i^G$ for $i = 1,2$. Then $\Hom_{S^G}(M_1, M_2)$ is in $\add_{S^G}(S)$ and so is a maximal \CM \ $S^G$-module.
\end{corollary}
\begin{proof}
Note that $P_1,P_2$ are free $S$-modules. So $\Hom_S(P_1,P_2)$ is a free $S$-module. Thus $\Hom_S(P_1,P_2)$ is a projective $S*G$-module. Notice
\[
\Hom_S(P_1,P_2)^G = \Hom_{S*G}(P_1,P_2) = W.
\]
It follows that $W \in \add_{S^G}(S)$.  We also note that as $S^G$ is a normal \CM \ domain a maximal \CM \ module is reflexive.  As the fixed point functor gives an equivalence of categories between $Ref(S*G)$ and $Ref(S^G)$ we get that 
\[
W = \Hom_{S*G}(P_1,P_2) \cong \Hom_{S^G}(M_1, M_2).
\]
The result follows.
\end{proof}
A consequence of the above result is
\begin{proposition}\label{mod-hom}
Let $\theta = \pi-g$ where $g \in (x_1,x_2)^2\cap S^G$ (Note $g = 0$ is also considered). Set $B = S^G/(\theta)$. Let $L_1,L_2 \in \add_{S^G}(S)$. Then
$$ \Hom_B\left(\frac{L_1}{\theta L_1}, \frac{L_2}{\theta L_2}\right) \cong \frac{\Hom_{S^G}(L_1,L_2)}{\theta \Hom_{S^G}(L_1,L_2)}.$$
\end{proposition}
\begin{proof}
We note that $B \cong (S/\theta)^G$ has finite representation type and so in particular is an isolated singularity, \cite{Aus-2} (also see \cite[4.22]{Yoshino}, \cite[Corollary 2]{HL}). As $\theta \neq 0$ it is $S^G$-regular. 
As $L_2$ is maximal \CM \ $S^G$-module, we get that $\theta$ is also $L_2$ regular.
We have an exact sequence $0 \rt L_2 \xrightarrow{\theta} L_2 \rt L_2/\theta L_2 \rt 0$. So we have an exact sequence
\begin{align*}
0 \rt &\Hom_{S^G}(L_1,L_2) \xrightarrow{\theta} \Hom_{S^G}(L_1,L_2) \rt \Hom_B(L_1/\theta L_1, L_2/\theta L_2) \\
&\Ext^1_{S^G}(L_1,L_2)    \xrightarrow{\theta} \Ext^1_{S^G}(L_1,L_2) \rt \Ext^1_B(L_1/\theta L_1, L_2/\theta L_2)
\end{align*}
Let $K = \ker( \Ext^1_{S^G}(L_1,L_2)    \xrightarrow{\theta} \Ext^1_{S^G}(L_1,L_2))$. By \ref{consequence-ht1} we get that $\Hom_{S^G}(L_1,L_2)$ is maximal \CM \ $S^G$-module. So it has $\depth = 3$. Also clearly
$\depth \Hom_B(L_1/\theta L_1, L_2/\theta L_2)  = 2$. It follows that $\depth K \geq 1$. Let $\m$ be the maximal ideal of $S^G$. We get that $\m \notin \Ass K$.

To prove our result it is sufficient to show that $\theta$ is a non-zero divisor of \\ $\Ext_{S^G}^1(L_1,L_2)$. Suppose if possible this is not true. Then $\theta \in \q$ for some $\q \in \Ass_{S^G} \Ext^1_{S^G}(L_1,L_2)$. Say
$\q = (0 \colon v)$ for some $v \in \Ext^1_{S^G}(L_1,L_2)$. So $\theta v = 0$. Thus $v \in K$. It follows that
$\q \in \Ass K$. So $\q \neq \m$. 

As $B$ is an isolated singularity we have that $\Ext^1_B(L_1/\theta L_1, L_2/\theta L_2)$ is a module of finite length. Localizing the above exact sequence at $\q$ we get
\[
0 \rt K_\q \rt \Ext^1_{S^G}(L_1,L_2)_\q    \xrightarrow{\theta} \Ext^1_{S^G}(L_1,L_2))_\q \rt 0.
\]
So $ \Ext^1_{S^G}(L_1,L_2))_\q =  \theta\Ext^1_{S^G}(L_1,L_2))_\q$. By Nakayama Lemma we get that \\ $ \Ext^1_{S^G}(L_1,L_2))_\q = 0$. So $K_\q = 0$, a contradiction. 
\end{proof}
A consequence of the above result is the following:
\begin{theorem}\label{ext-S}
(with hypotheses as above) Let $M_1, M_2$ be maximal CM $R^G$-modules and let
$\phi \colon M_1 \rt M_2 $ be a non-split epimorphism.
Then
\begin{enumerate}[\rm (1)]
\item
For $i = 1, 2$,
there exists $\widetilde{M_i} \in \add_{S^G}(S)$ with $\widetilde{M_i}\otimes_{S^G}R^G \cong M_i$.
\item
There exists $\widetilde{\phi} \in \Hom_{S^G}(\widetilde{M_1}, \widetilde{M_2})$ with
$\widetilde{\phi}\otimes_{S^G}R^G = \phi$.
\item
$\widetilde{\phi}$ is not a split epimorphism.
\item
For $i = 1,2$,
$M_i^* = \widetilde{M_i}/\pi \widetilde{M_i}$ are maximal \CM \ $A^G$-modules.
(Recall $A = k[[x_1,x_2]]$).
\item
Define $\phi^* = \widetilde{\phi}\otimes_{S^G}A^G.$ Then $\phi^* \colon M_1^* \rt M_2^*$
is not a split epi.
\end{enumerate}
\end{theorem}
\begin{proof}
(1) Let $P_0,P_1,\ldots, P_m$ be the complete list of non-isomorphic indecomposable projective $V[G]$-modules. Set $L_j = \psi(P_j)$ for $j = 0,\ldots,m$.
Then by \ref{basic},  $L_0,L_1,\cdots,L_m$ are all the non-isomorphic, indecomposable maximal \CM \ $R^G$-modules. So there exists integers $a_{i,j} \geq 0$ such that
\[
M_i \cong \bigoplus_{j=0}^m L_j^{a_{i,j}}, \quad \text{for} \ i = 1,2.
\]
 Define the following projective
$V[G]$-modules
\[
Q_i = \bigoplus_{j=0}^m P_j^{a_{i,j}}, \quad \text{for} \ i = 1,2.
\]
Set $\widetilde{M_i} = (F(Q_i))^G$ for $i = 1,2$ (here notation as in section \ref{c-f}). Clearly $\widetilde{M_i}\otimes_{S^G}R^G \cong M_i$.

(2) This follows from \ref{mod-hom}.

(3) If $\widetilde{\phi}$ is a split epi then it will follow that $\phi$ is a split epi, a contradiction.

(4) Recall $S^G/\pi S^G \cong A^G$, see \ref{mod-reg}. As $\pi$ is $S^G$-regular the result follows.

(5) Suppose if possible $\phi^*$ is a split epi. Consider the following exact sequence:
\[
0 \rt K \rt \widetilde{M_1} \xrightarrow{\widetilde{\phi}} \widetilde{M_2} \rt C \rt 0.
\]
As $\phi^*$ is epi we get that $C/\pi C = 0$. By Nakayama Lemma we get that $C = 0$.
Thus we have an exact sequence
\begin{equation}\label{in}
0 \rt K \rt \widetilde{M_1} \xrightarrow{\widetilde{\phi}} \widetilde{M_2} \rt 0.
\end{equation}
It follows that $K$ is a maximal \CM \ $S^G$-module.
Set $K^* = K/\pi K$. As $\phi^*$ is a split epi we get that $K^*\oplus M_2^* \cong M_1^*$. Set $\ov{Q_i} = k\otimes_VQ_i$ (here $Q_i$ is as in (1)).  By \cite[1.4 and 2.2]{Aus-1} there exists a projective $k[G]$-module $D^*$ with $K^* = (\ov{F}(D^*))^G$. Note we also have
$\ov{Q_2} \cong \ov{Q_1}\oplus D^*$. Let $D$ be the $V[G]$-projective cover of $D^*$.
Then $Q_2 \cong Q_1 \oplus D$.
 
Apply $-\otimes_{S^G}R^G$ to equation \ref{in}. We get
\begin{equation}\label{in-2}
0 \rt \ov{K} \rt M_1 \xrightarrow{\phi} M_2 \rt 0.
\end{equation}
As $\ov{K}$ is a maximal \CM \ $R^G$-module we have that $\ov{K} \cong \psi(W)$ for some projective $V[G]$-module $W$. We \\
\textit{Claim:} $W \cong D$.\\
Assume the claim for the moment. Note 
\[
M_1 \cong \psi(Q_1) = \psi(Q_2 \oplus D) \cong \psi(Q_2)\oplus \psi(D) \cong M_2 \oplus \ov{K}.
\] 
Thus the exact sequence \ref{in-2} is apparently split and so by Miyata's theorem (\cite{Miyata}, also see\cite[A3.29(a)]{eis}), we get that $\ref{in-2}$ is split. So $\phi$ is a split epi. This is a contradiction.

It remains to prove the claim. Let $W^* = k\otimes_V W$. In the ring $A^G/(\ov{f})$ we have that the modules $(\ov{F}(D^*))^G/\ov{f}(\ov{F}(D^*))^G$ and  $(\ov{F}(W^*))^G/\ov{f}(\ov{F}(W^*))^G$ are isomorphic. As $l \geq l_0(G)$ it follows that 
$(\ov{F}(D^*))^G \cong (\ov{F}(W^*))^G$. By  \cite[1.4 and 2.2]{Aus-1}  we have that $D^* \cong W^*$. As $D$ is the projective cover of $D^*$ and $W$ is a projective cover of $W^*$ we get that $D \cong W$.
\end{proof}

 \section{AR sequences}
 In this section we construct AR-sequences for $R^G$. We assume $f \in (x_1,x_2)^l\cap S^G$ with $l \geq l_0(G)$. Notice we are assuming that $G$ has no pseudo-reflections. 

\s Set $E = Vx_1\oplus Vx_2$ the free $V$-module with basis $x_1,x_2$. We write the Koszul complex of $S$ \wrt \ $(x_1,x_2)$ as
\begin{equation}\label{koszul}
0 \rt S\otimes_V\wedge^2 E \rt S\otimes_V E \rt S \rt V \rt 0.
\end{equation}
It is easy to see that \ref{koszul} is also an exact sequence of $S*G$-modules.
Let $P_0,P_1,\cdots,P_m$ be the complete list of indecomposable  projective $V[G]$-modules with $P_0 = V$. Applying the functor $-\otimes_V P_i$ to \ref{koszul} we obtain
\begin{equation}\label{koszul-P}
0 \rt S\otimes_V\left(\wedge^2 E\otimes_V P_i\right))  \rt S\otimes_V( E \otimes_V P_i) \xrightarrow{\widehat{p_i}} S\otimes_V P_i \xrightarrow{\epsilon_i} P_i \rt 0,
\end{equation}
which gives a projective resolution of the $S*G$-module $P_i$ as
by the following result we get that $\wedge^2 E\otimes_V P_i$ and $E \otimes_V P_i$ are projective $V[G]$-modules.
\begin{proposition}\label{ARS}
Let $X$ be a projective $V[G]$-module and let $Y$ be a $V[G]$-module which is free as a $V$-module. Then $X\otimes_V Y$ is a projective $V[G]$-module.
\end{proposition}
\begin{proof}
The proof in page 82 of Proposition 3.1 in \cite{ARS}  generalizes.
\end{proof}

Taking invariants in equation \ref{koszul-P} we obtain an exact sequence of $S^G$-modules which we denote as follows:
\begin{equation}\label{eqn-P-s}
0 \rt \tau(\widetilde{L_i}) \rt\widetilde{ E_i}\xrightarrow{ \widetilde{p_i}} \widetilde{L_i}  \rt P_i^G \rt 0. 
\end{equation}
Note that $P_i^G = V$ if $i = 0$ and $P_i^G = 0$ otherwise. This follows from the 
the following result.
\begin{proposition}
Let $X$ be a finitely generated $V[G]$ module. Then
\begin{enumerate}[\rm (1)]
\item
$X^G$ is a $V$-direct summand of $X$.
\item
If $X$ is a projective $V[G]$-module then $X^G$ is a projective submodule of $X$.
\item
If $X$ is an indecomposable projective $V[G]$-module then $X^G = X$ or $X^G = 0$.
\end{enumerate}
\end{proposition}
\begin{proof}(Sketch)
(1) Use the Reynolds operator to get the result.

(2)The usual proof of Maschke's theorem \cite[6.1]{Lam} can be adapted to prove this result.

(3) Easily follows form (2).
\end{proof}

\s Let $\theta = \pi - f$ where $f\in (x_1,x_2)^l\cap S^G$ and $f \notin \pi S$. We assume that $l \geq l_0(G)$. Note that as $\tau(\widetilde{L_i}), \widetilde{E_i} , \widetilde{L_i} \in \add_{S^G}(S)$ they are maximal \CM \ $S^G$-modules. So $\theta$ is  $\tau(\widetilde{L_i})\oplus \widetilde{E_i} \oplus \widetilde{L_i}$-regular. Set $L_i = \widetilde{L_i}/\theta\widetilde L_i, E_i = \widetilde{E_i}/\theta \widetilde{E_i}$ and $\tau(L_i) = \tau(\widetilde{L_i})/\theta \tau(\widetilde{L_i})$. Then $L_i, E_i$ and $\tau(L_i)$ are maximal \CM \ $R^G$-modules.  Also note the action of $\theta$ on $V$ is same as that of $\pi$. So $\theta$ is $V$-regular and $V/\theta V = k$. Thus we have the following exact sequences:
\begin{align}
\label{ar-0}0 &\rt \tau(L_0) \rt E_0 \xrightarrow{p_0} L_0 \rt k \rt 0 \ \text{and} \\
\label{ar-1}0 &\rt \tau(L_i) \rt E_i \xrightarrow{p_i} L_i  \rt 0 \ \text{for} \ i>0. 
\end{align}
Note that as $l \geq l_0(G)$ we have that $L_i = \psi(P_i)$ for $i = 0,1,\ldots,m$ are precisely the indecomposable maximal \CM \ $R^G$-modules. Also note that $L_0 = R^G$.
\begin{remark}(with notation as in \ref{eqn-P-s}) Set $(-)^* = (-)\otimes_V k  = (-)\otimes_{S^G} A^G$. Then note that we have the following exact sequences 
of $A^G$-modules:
\begin{align}
\label{ar-0-s}0 &\rt \tau(L_0^*) \rt E_0^* \xrightarrow{p_0^*} L_0^* \rt k \rt 0 \ \text{and} \\
\label{ar-1-s}0 &\rt \tau(L_i^*) \rt E_i^* \xrightarrow{p_i^*} L_i^*  \rt 0 \ \text{for} \ i > 0. 
\end{align}
Notice $L_i^*$ for $i = 0,1,\ldots,m$ are precisely the indecomposable maximal \CM \ $A^G$-modules. Also note that $L_0^* = A^G$. By a result due to Auslander, \cite[p.\ 516, 517]{Aus-1}, we have that if 
 $L^*$ is a maximal \CM \ $A^G$-module and if $\phi^* \colon L^* \rt L_i^*$ is an $A^G$-homomorphism which is not a split epimorphism then there exists an $A^G$ homomorphism $\lambda^* \colon L^* \rt E_i^*$ with $\phi^* = p_i^* \circ \lambda^*$.\\
\end{remark}

The following is the main result of this section.
\begin{theorem}\label{ar-main}(with hypotheses as above). For any $i (0 \leq i \leq m)$ the sequences \ref{ar-0} and \ref{ar-1} satisfy the following condition:\\
If $L$ is a maximal \CM \ $R^G$-module and if $\phi \colon L \rt L_i$ is an $R^G$-homomorphism which is not a split epimorphism then there exists an $R^G$ homomorphism $\lambda \colon L \rt E_i$ with $\phi = p_i \circ \lambda$.\\
In particular for $i \neq 0$ the sequence \ref{ar-1} is the AR sequence ending in $L_i$.
\end{theorem}
We first prove the following consequence of Theorem \ref{ext-S}.
\begin{lemma}\label{c-l}
(with hypotheses as above) If $\phi\colon L \rt L_i$  is  not a split epimorphism then there exists $R^G$ homomorphisms  $\lambda_0 \colon L \rt E_i$ and $\phi_1 \colon L \rt L_i$ with $\phi = p_i \circ \lambda_0 + \pi \phi_1$. 
\end{lemma}
 \begin{proof}
 Let $\widetilde{L_i}, \widetilde{L}, \widetilde{\phi}, L_i^*, L^*$ and $\phi^*$ be as in Theorem \ref{ext-S}. So $\phi^* \colon L^* \rt L_i^*$ is not a split epimorphism. Therefore there exists an $A^G$-homomorphism $\lambda_0^* \colon L^* \rt E_i^*$ such that $\phi^* = p_i^* \circ \lambda_0^*$. By \ref{mod-hom} there exists an $S^G$ linear map $\widetilde{\lambda_0} \colon \widetilde{L} \rt \widetilde{E_i}$ with 
 $\widetilde{\lambda_0}\otimes A^G = \lambda_0^*$. Again from \ref{mod-hom} there exists an $S^G$ linear map $\widetilde{\phi_1} \colon \widetilde{L} \rt \widetilde{L_i}$ such that $\widetilde{\phi} = \widetilde{p_i}\circ \widetilde{\lambda_0} + \pi \widetilde{\phi_1}$. Going mod $\pi-f$ we get our result.
 \end{proof}
 We now give
 \begin{proof}[Proof of Theorem \ref{ar-main}]
 By Lemma \ref{c-l} there exists $R^G$ homomorphisms  $\lambda_0 \colon L \rt E_i$ and $\phi_1 \colon L \rt L_i$ with $\phi = p_i \circ \lambda_0 + \pi \phi_1$. 
 We consider the following two cases:
 
 \textit{Case 1: } $L_i$ is not a summand of $L$.
 Then note that $\phi_1$ is also not a split epi. So there exists  $\lambda_1 \colon L \rt E_i$ and $\phi_2 \colon L \rt L_i$ with $\phi_1 = p_i \circ \lambda_1 + \pi \phi_2$. Thus $\phi = p_i\circ(\lambda_0 + \pi \lambda_1) + \pi^2\phi_2$. Again $\phi_2$ is not a split epi. Iterating this process we get for all $n \geq 0$ homomorphisms $\lambda_n \colon L \rt E_i$ and $\phi_{n+1} \colon L \rt L_i$ such that $\phi_{n+1}$ is not a split epi and 
 \[
 \phi = p_i \circ \left( \sum_{j=0}^{n}\pi^j\lambda_j \right) +  \pi^{n+1} \phi_{n+1}.
\]   
Set 
$$\lambda =  \sum_{j =0}^{\infty}\pi^j\lambda_j. $$
Then as $R^G$ is complete and $\pi \in \n$ the maximal ideal of $R^G$ it follows that 
$\lambda \colon L \rt E_i$ is $R^G$-linear. Clearly $\phi = p_i \circ \lambda$. 

 \textit{Case 2: } $L_i$ is  a summand of $L$. Set $L = L_i^r \oplus K$ where $L_i$ is not a summand of $K$. 
 
 We decompose $\phi = (\phi_i, \phi_K)$. Note that $\phi_i$ and $\phi_K$ are not split epi's. By Case 1, $\phi_K$ already has a lift. Thus it suffices to consider the case when $L = L_i^r$. 
 
 If $L = L_i^r$ then note $\phi = (\phi_1,\phi_2,\cdots, \phi_r)$. Clearly $\phi_i$ are not split epi for each $i$. Thus it suffices to consider the case when $r = 1$. That is $L = L_i$. 
 By Lemma \ref{c-l} there exists $R^G$ homomorphisms  $\lambda_0 \colon L_i \rt E_i$ and $\phi_1 \colon L_i \rt L_i$ with $\phi = p_i \circ \lambda_0 + \pi \phi_1$. 
 
 \textit{Claim:}
  the multiplication map $\mu_\pi \colon L_i \rt L_i$ can be factored as $p_i \circ \delta$ where $\delta \colon L_i \rt E_i$.
 
 Assume the claim for the moment. Then we have 
 \begin{align*}
 \phi &= p_i \circ \lambda_0 + \pi \phi_1, \\
   &= p_i \circ \lambda_0 + p_i \circ \delta \circ \phi_1, \\
   &= p_i\circ \left(\lambda_0 + \delta \circ \phi_1 \right).
 \end{align*}
 Thus $\phi$ can be lifted.
 
We now prove our Claim. The essential point is that $\mu_\pi = \mu_f$ as $R^G = S^G/(\pi -f)$. 
 We consider the exact sequence \ref{koszul-P}. Notice as $f \in S^G$ we get that $f \in Z(S*G)$ the center of $S*G$. So the multiplication map  $\mu_f \colon S\otimes_V P_i \rt S\otimes_V P_i$ is $S*G$-linear. As $f \in (x_1,x_2)$ we get that $\epsilon_i \circ \mu_f = 0$. As $S\otimes_V P_i$ is a projective $S*G$-module we have a lift $\widehat{\delta} \colon S\otimes_V P_i \rt S\otimes_V(E\otimes_V P_i)$. Taking invariants and going mod $\pi -f$ yields the Claim.
 \end{proof}

A consequence of the above result is the following:
\begin{theorem}\label{AR-quiver}
(with hypotheses as in Theorem \ref{ar-main}) The AR quiver of $R^G$ coincides with the McKay graph $Mc(k^2, G)$.
\end{theorem}
\begin{proof}
Let $P_0, P_1, \cdots, P_m$ be the complete list  of non-isomorphic indecomposable projective $V[G]$-modules with $P_0 = V$. If $Y$ is a projective $V[G]$ module then $Y = P_0^{a_0}\oplus P_1^{a_1}\oplus \cdots \oplus P_m^{a_m}$ for some $a_i \geq 0$. Set $\mult_i(Y) = a_i$.

By Proposition \ref{ARS} we get that $V^2\otimes_V P_j$ is a projective $V[G]$-module. Let $T(V^2,G)$ be the oriented  graph whose vertices are $P_0,\ldots,P_m$ and  there are $c_{i,j}$ arrows from $P_i$ to $P_j$ if $\mult_i(V^2\otimes_V P_j) = c_{i,j} \neq 0$.

The proof in \cite[10.14]{Yoshino} generalizes to our situation (we have to use Theorem \ref{ar-main} also) and so the AR-quiver of $R^G$ is isomorphic to $T(V^2,G)$.

We now note that $W_i = \ov{P_i} = P_i/\pi P_i$ is a complete list of non-isomorphic irreducible representations of $G$ over $k$, see \ref{perfect}. If $X$ is a representation of $G$ over $k$ then $X = W_0^{a_0}\oplus W_1^{a_1}\oplus \cdots \oplus W_m^{a_m}$ for some $a_i \geq 0$. Set $\ov{\mult_i}(X) = a_i$. Recall that the McKay graph $Mc(k^2,G)$ is defined to be an oriented graph whose vertices are  $W_0,\cdots,W_m$ and there are 
$c_{i,j}$ arrows from $W_i$ to $W_j$ if $\ov{\mult_i}(k^2\otimes_V W_j) = c_{i,j} \neq 0$.

In general if $Y$ is a projective $V[G]$ then $Y$ is a projective cover of $Y/\pi Y$.
Also clearly $\mult_i Y = \ov{\mult_i}(Y/\pi Y)$. Note that 
\[
\left(V^2\otimes_V P_j \right)\otimes_V V/(\pi) \cong k^2\otimes_k W_j.
\]
Thus there are $\mu$ arrows from $P_i$ to $P_j$ in $T(V^2,G)$ if and only if there are $\mu$ arrows from $W_i$ to $W_j$ in $Mc(k^2,G)$. Thus $T(V^2,G) \cong Mc(k^2,G)$.
Our result follows
\end{proof}

\section{an Explicit example- The invarinats of Klein group $A_n$}
In this section we give an explicit example, the invariants of Klein group $A_n$ to illustrate our results. For simplicity we assume that the characteristic $p$ of the residue field $k$ is not equal to $2$.  Throughout $S = V[[x_1,x_2]]$ and $A = k[[x_1,x_2]]$. 

Throughout $\zeta_n$ will denote a primitive $n^{th}$ root of unity in $k$ when $n$
is coprime to $p$. By $\theta_n$ we mean an element of $V$ with $\ov{\theta_n} = \zeta_n$ and $\theta^n_n = 1$. By Hensel's lemma such a $\theta_n$ does exist. It can be  easily seen, for instance from \ref{inject-k},  that $\theta_n$ is unique.

(I) $(A_n)$  Cyclic group of order $n+1$. Here $n \geq 1$ and
\[
G = < \begin{pmatrix}
          \theta_{n+1}  & 0 \\
             0 & \theta_{n+1}^{-1}          
                \end{pmatrix} >
\]
Note we are assuming $p \nmid n+1$ and $p \neq 2$. 
We first prove
\begin{theorem}\label{KA}
(with hypotheses as above)
$$ S^G \cong V[[x,y,z]]/(x^2 + y^{n+1}+ z^2).$$
Furthermore if $f \in (x,y,z)^{3(n+1)}$ and $f \notin \pi S^G$ then $\ov{f} \in (u_1^2, u_2^2)$ for some efficient system of parameters $u_1, u_2$ of $A^G$. Thus $R^G = S^G/(\pi -f)$ will have AR quiver isomorphic to the McKay graph of $G$. 
\end{theorem}
\begin{proof}We note that $x_1x_2, x_1^{n+1}, x_2^{n+1} \in S^G$. 

\textit{Claim-1:} $S^G =  V[[x_1x_2, x_1^{n+1}, x_2^{n+1}]]$.\\
Set $T =  V[[x_1x_2, x_1^{n+1}, x_2^{n+1}]] \subseteq S^G$.
Put $v_1 = x_1x_2, v_2 = x_1^{n+1}, v_3 =  x_2^{n+1}$ and $u_i = \ov{v_i}$ the image in $A$ for $i = 1,2,3$.  It is well-known that $A^G = k[[u_1,u_2,u_3]]$.

Let $x \in S$. Set $x = \sum_{i\geq 0} x_i$ where $x_i$ is homogeneous of degree $i$. Then $x \in S^G$ if and only if $x_i \in S^G$ for all $i \geq 0$.  We first show $x_i \in T$ for all $i$.

 For $\alpha = (\alpha_1,\alpha_2,\alpha_3) \in \mathbb{N}^3$ set $|\alpha| = 2\alpha_1 + (n+1)\alpha_2 + (n+1)\alpha_3$. Also set $u^\alpha = u_1^{\alpha_1}u_2^{\alpha_2}u_3^{\alpha_3}$. Define $v^\alpha$ analogously. Notice $\ov{x_i} \in A^G$.
\[
\ov{x_i} = \sum_{|\alpha| = i } \ov{a_{0,\alpha}}u^\alpha 
\]
for some $ \ov{a_{0,\alpha}} \in k$ for all $\alpha$ with $|\alpha| = i$. 
Set $$z_0 =  \sum_{|\alpha| = i } a_{0,\alpha}v^\alpha.$$
Then
$x_i = z_0 + \pi x_{i,1}$ for some $x_{i,1} \in S$. Notice $x_{i,1} \in S^G$. Also note that $x_{i,1}$ is homogeneous of degree $i$.
Iterating the above procedure we get
$x_{i,1} = z_1 + \pi x_{i,2}$ for some $x_{i,2} \in S^G$ and
$$z_1 =  \sum_{|\alpha| = i } a_{1,\alpha}v^\alpha.$$ Thus for all $m \geq 1$ we obtain relation 
 $x_{i,m} = z_m + \pi x_{i,m+1}$ for some $x_{i,m+1} \in S^G$ homogeneous of degree $i$ and
$$z_m =  \sum_{|\alpha| = i } a_{m,\alpha}v^\alpha.$$ 
We also have
$$x_i = z_0 + \pi z_1 + \pi^2 z_2 + \cdots + \pi^{m}z_m + \pi^{m+1}x_{i,m+1}.$$
Notice that for all $\alpha$ with  $|\alpha| = i$ we get that
$$ a_\alpha^{(i)} =  a_{0,\alpha} + \pi a_{1,\alpha} + \cdots +   \pi^m  a_{m,\alpha} + \cdots \in V$$
 So we get that
 \begin{align*}
 z^{(i)} &= z_0 + \pi z_1 + \cdots + \pi^m z_m + \cdots \\
    &= \sum_{|\alpha| = i } a_{\alpha}^{(i)}v^\alpha \in T
 \end{align*}
 Clearly $x_i = z^{(i)}$. Thus $x_i \in T$.

If $x = \sum_{i \geq 0}x_i \in S^G$ with $x_i$ homogeneous of degree $i$ then note that
\[
x = \sum_{i \geq 0}\left( z^{(i)}  \right) = \sum_{i \geq 0}\left( \sum_{|\alpha| = i } a_{\alpha}^{(i)}v^\alpha   \right) \in T
\]

\textit{Claim-2:} $S^G \cong V[[x,y,z]]/(x^2 + y^{n+1}+ z^2)$. \\
As $p \nmid 2(n+1)$ note that $\theta_4, \theta_{2(n+1)} \in V$. Define $\beta = \theta_{2(n+1)}v_1$. Also define $\alpha, \gamma$ by the formula $\alpha + \theta_4 \gamma  = v_2$ and $\alpha - \theta_4 \gamma = v_3$. Then notice $S^G = V[[\alpha, \beta, \gamma]]$. Furthermore as $v_1^{n+1} = v_2v_3$ we obtain $\alpha^2 + \beta^{n+1} + \gamma^2 = 0$.

Set $B = V[[x,y,z]]/(x^2 + y^{n+1}+ z^2)$. Note we have an obvious surjective map $B \rt S^G$. Note $B$ is \CM \ of dimension $3$. It suffices to show that $B$ is a domain. Notice
$B/\pi B \cong k[[x,y,z]]/(x^2 + y^{n+1}+ z^2) \cong A^G$ is a domain of dimension $2$. Thus $\pi B$ is a prime ideal of height one. Let $\q$ be a minimal prime of $B$ contained in $\pi B$. Let $a \in \q$. Then $a = \pi b$ for some $b \in B$. As $\pi \notin \q$ we get that $b \in q$. Thus $\q = \pi \q$. By Nakayama Lemma $\q = 0$. So $B$ is a domain.

Notice $A^G \cong  k[[x,y,z]]/(x^2 + y^{n+1}+ z^2)$. We now assert

\textit{Claim 3:} $(x,z)$ is an efficient system of parameters for $A^G$.\\
Consider the subring $T_1 = k[[x,y]]$ of $A^G$. As $A^G \cong T[Z]/(p(Z))$
 where $p(Z) = Z^2 + x^2 + y^{n+1}$ we get that dimension of $T_1$ is $2$. Thus $T_1$ is regular local. Let $\mathcal{N}_1$ be the Noetherian differnt of $A^G$ over $T_1$, see \cite[6.6]{Yoshino}. By \cite[6.13]{Yoshino}, $p'(z) \in \mathcal{N}_1$. Thus $2z \in \mathcal{N}_1$.
 
 Similarly by considering the regular subring $T_2 = k[[y,z]]$ of $A^G$ we get that
 $2x \in \mathcal{N}_2$, the Noetherian differnt of $A^G$ over $T_2$. Furthermore as $p \neq 2$ we get that $(x,z)$ is an efficient system of parameters of $A^G$.
 
 Finally note that $(x,y,z)^{n+1} \subseteq (x,z)$. Thus
 $(x,y,z)^{3(n+1)} \subseteq (x,z)^3 \subseteq (x^2,z^2)$.

\begin{remark}
In a similar but tedious way one can analyze the invariant rings of the Klein groups $D_n, E_6, E_7, E_8$.
\end{remark}

\end{proof}

\end{document}